\documentclass[11pt]{article}

\usepackage{amsfonts}
\usepackage{amsmath}
\usepackage{amsthm}
\usepackage{mathrsfs}

\usepackage{algorithm}
\usepackage{algorithmic}

\usepackage{graphicx}
\usepackage{url}
\usepackage{hyperref}

\usepackage{listings}
\usepackage{nameref}
\usepackage{tabularray}
\usepackage{adjustbox}

\usepackage[square,numbers]{natbib}
\usepackage[dvipsnames]{xcolor}
\usepackage[most]{tcolorbox}

\usepackage{jair, rawfonts}

\usepackage{tikz}
\usepackage{tikz-cd}
\usetikzlibrary{matrix,arrows}

\newtcolorbox{definition}[1]
{
  title={#1},
  coltitle = Green,  
  colframe=Green!15!white,
  colback=Green!10!white,
  boxrule=0mm,
  sharpish corners,
  breakable
}

\newtcolorbox{lemma}[1]
{
  title={#1},
  coltitle = Blue,  
  colframe=Blue!10!white,
  colback=Blue!5!white,
  boxrule=0mm,
  sharpish corners,
  breakable
}

\newtcolorbox{corollary}[1]
{
  title={#1},
  coltitle = Blue,  
  colframe=Blue!5!white,
  colback=Blue!2!white,
  boxrule=0mm,
  sharpish corners,
  breakable
}

\newtcolorbox{theorem}[1]
{
  title={#1},
  coltitle = Purple,  
  colframe=Purple!15!white,
  colback=Purple!10!white,
  boxrule=0mm,
  sharpish corners,
  breakable
}

\newtcolorbox{proofbox}
{
  coltitle = Black,  
  colframe=Black!5!white,
  colback=Black!2!white,
  boxrule=0mm,
  sharpish corners,
  breakable
}

\firstpageno{1}

\begin{document}

\author{
  Alice Petrov
}
\title{The Essence of de Rham Cohomology}

\setcounter{page}{0}
\setcounter{secnumdepth}{1}

\newpage

\maketitle


\begin{abstract}
The study of differential forms that are closed but not exact reveals important information about the global topology of a manifold, encoded in what are called the de Rham cohomology groups $H^k(M)$, named after Georges de Rham (1903–1990).
This expository paper is an explanation and exploration of de Rham cohomology and its equivalence to singular cohomology. We present an intuitive introduction to de Rham cohomology and discuss four associated computational tools: the Mayer-Vietoris theorem, homotopy invariance, Poincaré duality, and the Künneth formula. We conclude with a statement and proof of de Rham's theorem, which asserts that de Rham cohomology is equivalent to singular cohomology.
\end{abstract}

\tableofcontents



\section{Introduction}

In mathematics, cohomology offers a unifying perspective. Given a mathematical object \(\mathcal{X}\) of dimension $d$, it's often possible to naturally construct a graded real vector space 
\[
\mathcal{A}(\mathcal{X}) = \bigoplus_{k=0}^d \mathcal{A}^k(\mathcal{X})
\]
known as the cohomology of \(\mathcal{X}\). 
When two objects of the same type, \(\mathcal{X}\) and \(\mathcal{Y}\), are related, these relationships are typically reflected in their respective cohomologies \cite{huh2022combinatorics}. Historically, topologists and geometers have been the primary beneficiaries of this viewpoint. Today, it will be us.

The conventional principles of geometry and calculus fail over ``unconventional'' topological spaces, and the extent to which they fail tells us something about the properties of the space.
In geometry, this is captured by homology, and in calculus, by cohomology.
Here's the kicker: despite their different names, homology and cohomology are twins. They both measure precisely the same holes in a space: this is formulated in de Rham's theorem, which we state and prove in Section \ref*{section:derhamtheorem}. 

The two constructions with which we will primarily be concerned are singular cohomology and de Rham cohomology. The singular cohomology groups are obtained from a simple dualization in the definition of singular homology. The de Rham cohomology groups are the quotients of ``closed forms'' and ``exact forms.'' Closed forms are differential forms whose exterior derivative is zero. Exact forms are differential forms obtained by applying the exterior derivative.

This exposition assumes that the reader is familiar with singular (co)homology, differential forms, and smooth manifolds. Suitable introductions to these topics can be found in \cite{hatcher2005algebraic}, \cite{guillemin2019differential}, and \cite{lee2012smooth}, respectively. 
The remainder of this paper is organized as follows. We begin with an intuitive introduction to de Rham cohomology, and the properties we expect to capture, followed by a more structured introduction in Section 2. We then, in great generality, cover some basic results and tools for computation in Section 3. We prove one of these tools, the Mayer Vietoris theorem, in Section 4. To fully convince the reader that these computations do indeed capture the properties we expect them to (the same properties as homology), Section 5 is dedicated to a statement and proof of de Rham's theorem.

\section{De Rham Cohomology}
\label{section:derham}

De Rham cohomology establishes a connection between the differential forms defined on a manifold and its global topological properties. 
In the following section, we provide some intuition for how this works and hopefully convince you, the reader, that it is worth caring about. The main reference used is \cite{bachman2012geometric}, Section 7.7.

\subsection{Preamble}

Consider the punctured plane $M = \mathbb{R}^2 - \{(0, 0)\}$. Every point $p \in M$ is contained in an open set identical to an open set in $\mathbb{R}^2$. This implies that all of the local properties of $\mathbb{R}^2 - \{(0, 0)\}$ are equivalent to those of $\mathbb{R}^2$. The presence of a hole in $M$ then represents a global topological property, it says something about the entire manifold.
\begin{center}
    \includegraphics[width=0.5\textwidth]{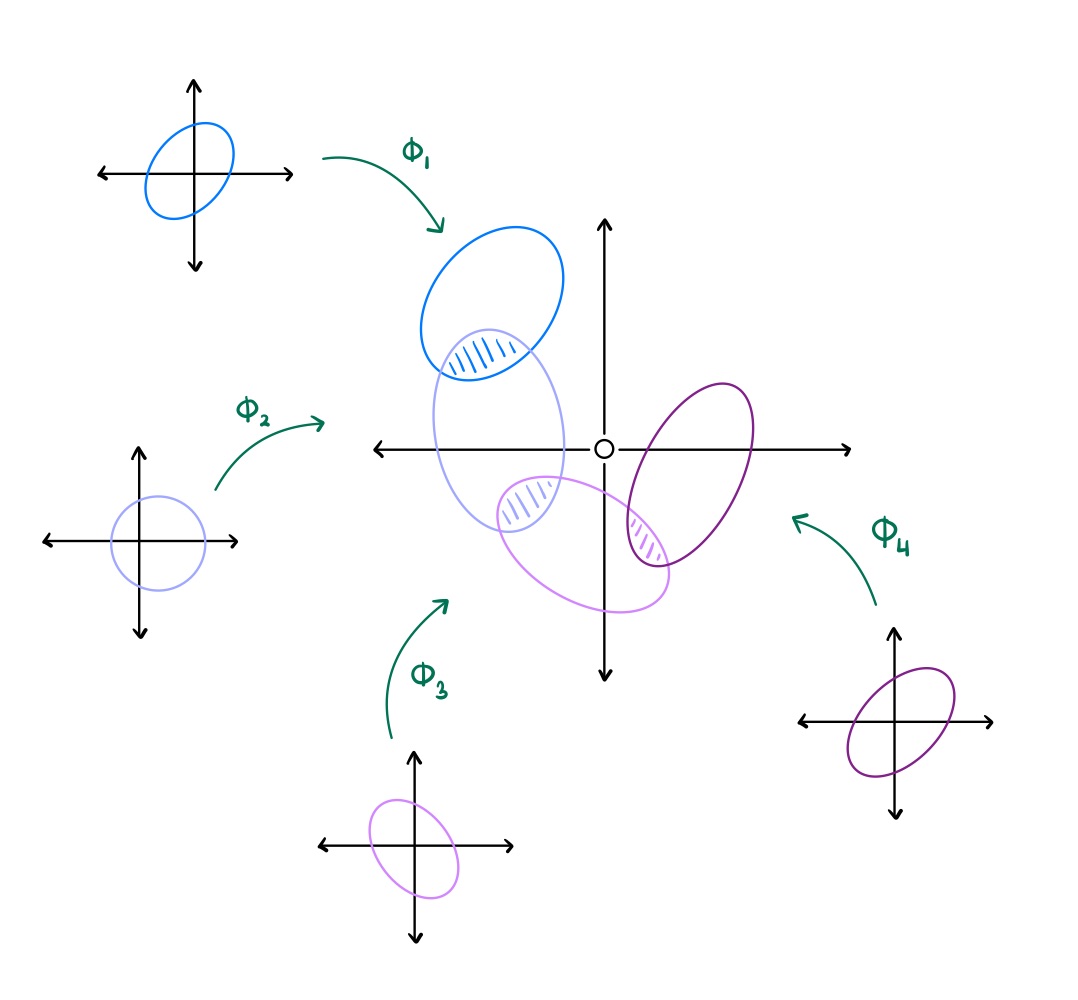}
\end{center}
To identify these global properties we turn to differential forms, and in the process, we encounter forms that are ``interesting'' and forms that are not. A differential form \(\omega\) on a manifold \(M\) is said to be closed if \(d\omega = 0\) and exact if \(\omega = d\tau\) for some differential form \(\tau\) of degree one less than $\omega$.
The interesting forms are those whose exterior derivative equals zero, namely ``closed'' forms.

First, let's demonstrate that closed forms do not provide interesting \textit{local} information by considering the closed 1-form $\omega_0$ defined on $\mathbb{R}^2 - \{(0, 0)\}$ by:
\[
\omega_0 = -\frac{y}{x^2 + y^2}dx + \frac{x}{x^2 + y^2}dy
\]
Looking locally at \(M\), we can integrate \(\omega_0\) over some small closed curve \(C\) that lies entirely inside an open subset \(U_i\); that is, over a ``local'' curve.
\begin{center}
    \includegraphics[width=0.3\textwidth]{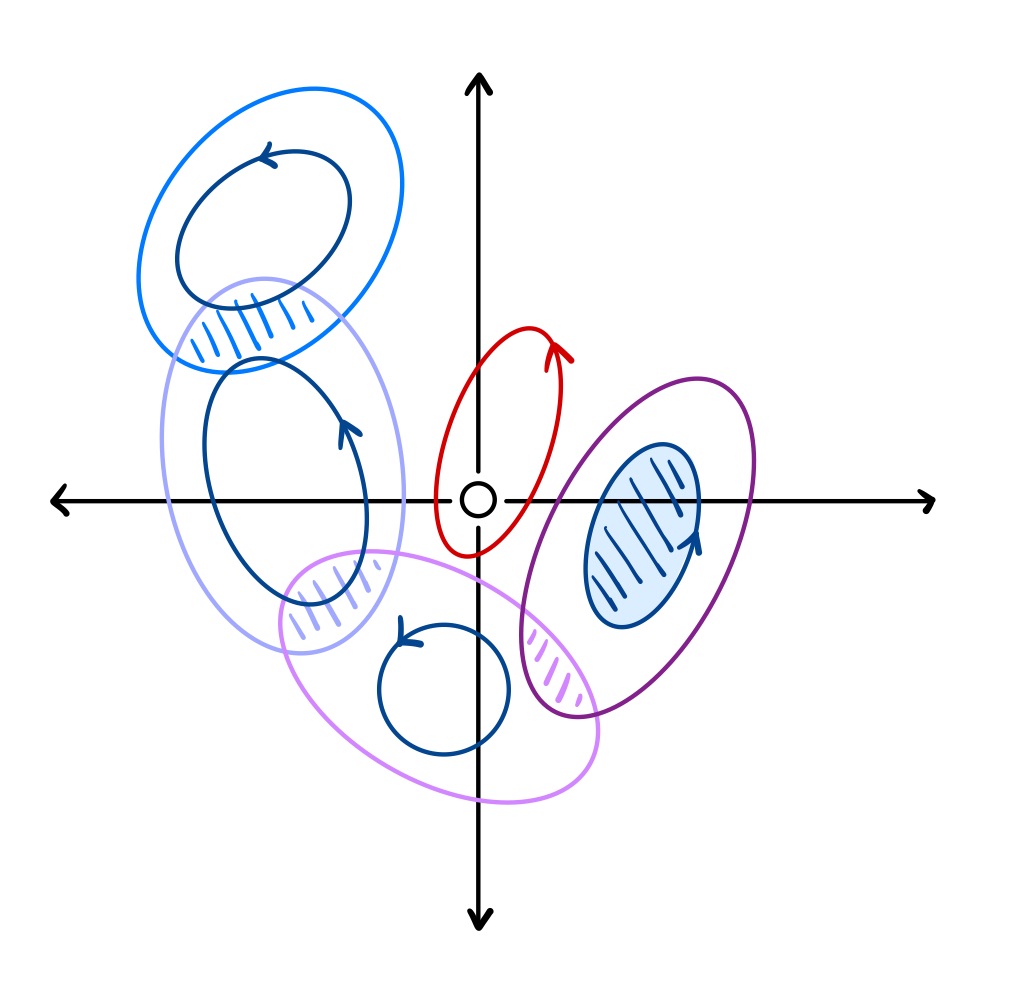}
\end{center}
Since the local curve \(C\) is the boundary of some disk \(D\), we have \(\partial C = 0\) (i.e., \(C\) is a closed 1-chain).
Then, Stokes' Theorem says
\[
\int_C \omega_0 = \int_D d\omega_0 = \int_D 0 = 0
\]
In a suitably small region within any manifold, every closed 1-chain forms the boundary of a disk. Consequently, the integration of closed 1-forms over these ``local'' curves does not convey any local information.

What happens when we integrate a closed 1-form \(\omega_0\) along a curve that goes around the missing point at the origin?
Firstly, this curve is no longer contained in a single open subset.
If we have both a closed 1-form \(\omega_0\) and a closed 1-chain \(C\) such that
\[
\int_C \omega_0 \neq 0
\]
then we can deduce that \(C\) doesn't bound a disk. Indeed, there must be something strange going on inside that curve \(C\). The existence of such a 1-chain provides valuable insight into the \textit{global} topological properties of \(M\). If \(C\) isn't the boundary of some disk, then it encloses some hole.

We will be using differential forms, and in particular, the relation between closed and exact forms, to determine global topological information about manifolds. Now, why should we particularly care about exact forms? On their own, exact forms do not provide especially interesting information.

For example, suppose that we have a 1-form \(\omega_1\) that is the derivative of some 0-form \(f\) (i.e., \(\omega_1 = df\)). That is, \(\omega_1\) is exact. 
Let \(C\) be a closed 1-chain. Whether or not \(C\) encloses a hole, we can divide $C$ into two segments by selecting two points, \(p\) and \(q\), on \(C\). Then \(C = C_1 + C_2\), where \(C_1\) goes from \(p\) to \(q\) and \(C_2\) goes from \(q\) back to \(p\). 
\begin{center}
    \includegraphics[width=0.3\textwidth]{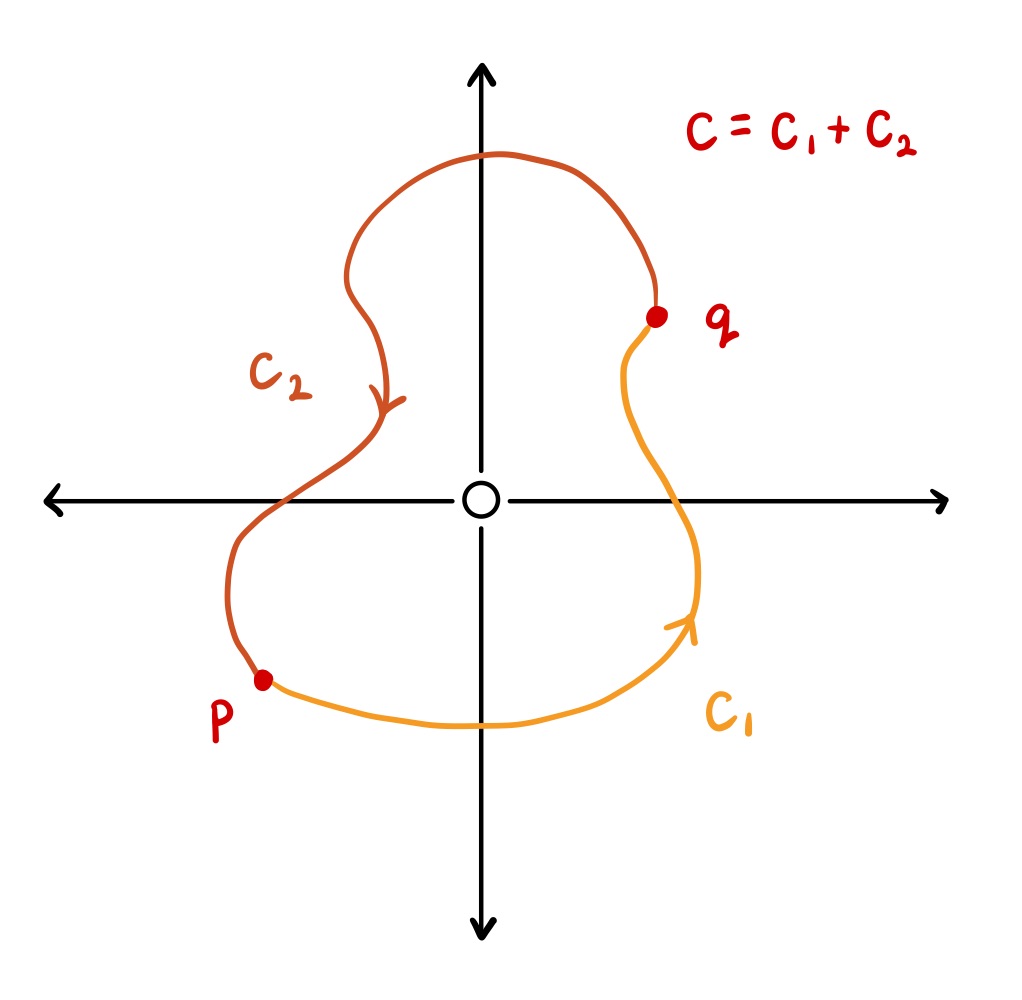}
\end{center}
Then,
\begin{align*}
    \int_C \omega_1 &= \int_{C_1 + C_2} \omega_1 \\
    &= \int_{C_1} \omega_1 + \int_{C_2} \omega_1 \\
    &= \int_{C_1} df + \int_{C_2} df \\
    &= \int_{p-q} f + \int_{q-p} f \\
    &= 0
\end{align*}
So integrating an exact form over a closed 1-chain always results in zero.

Since \(d^2 = 0\) (\cite{guillemin2019differential}, Properties 2.4.2), every exact form is closed. 
If our manifold were \( \mathbb{R}^n \), then, by the Poincaré lemma (\cite{guillemin2019differential}, Lemma 2.4.16), every closed form would also be exact. This implies that the set of closed forms would be identical to the set of exact forms. Therefore, there would be no curve \( C \) such that for the closed form \( \omega \), we would have \( \int_C \omega = 0 \). This shouldn't be surprising, as we know that there are no holes in \( \mathbb{R}^n \).

What we are truly interested in is the scenario where the set of closed forms differs from the set of exact forms. In general, not every closed form is exact. We seek some way to measure the difference between these two sets.

To address this, we introduce an equivalence relation.
We restrict ourselves to closed forms, but we consider two of them to be ``the same'' if their difference is an exact form. The set which we end up with is the de Rham cohomology of the manifold. 
\begin{center}
    \includegraphics[width=0.6\textwidth]{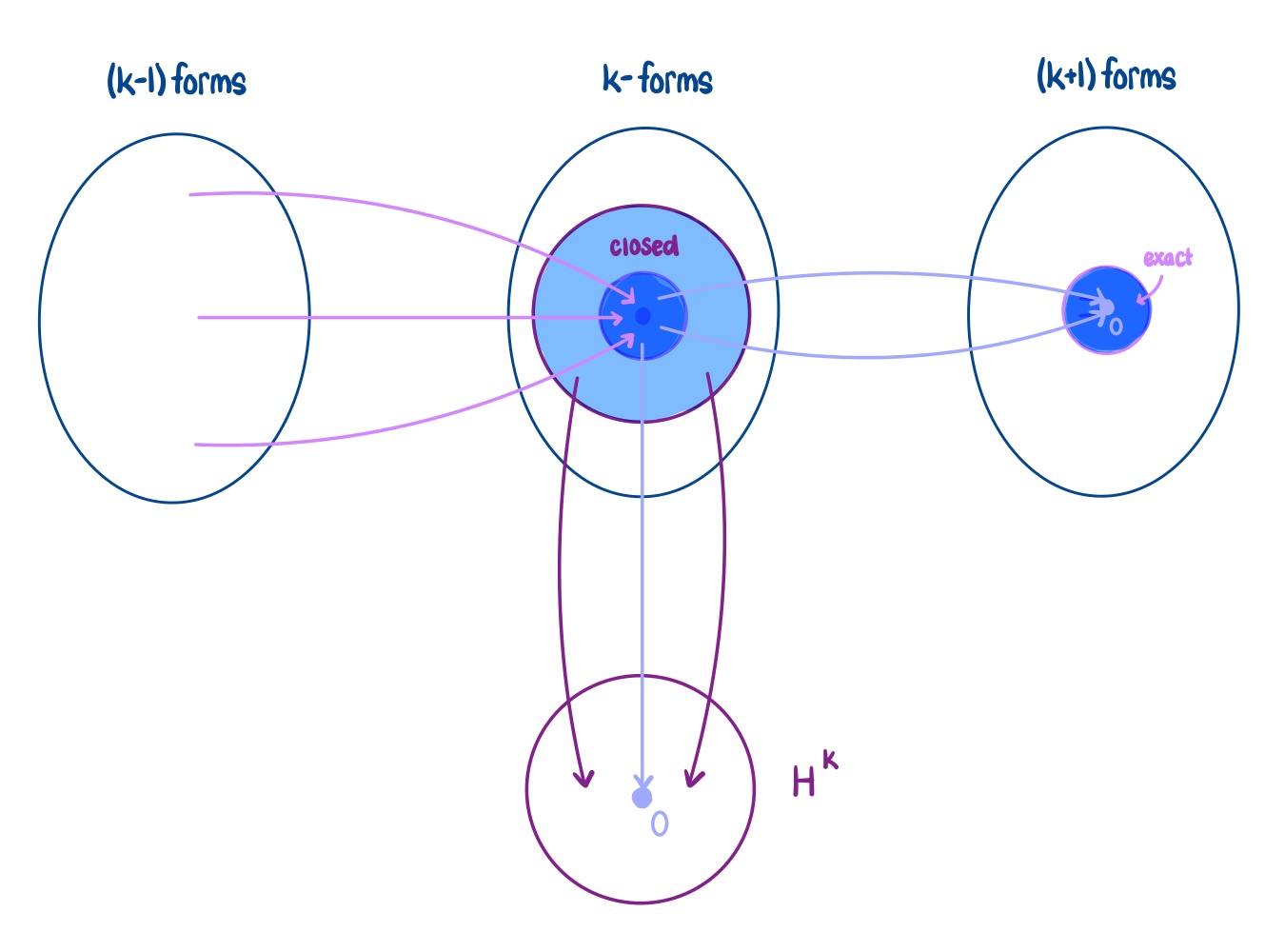}
\end{center}
It's worth noting that the difference between an exact form and a form that always evaluates to zero is another exact form. Hence, every exact form corresponds to zero in de Rham cohomology.
For each \(k\), the set \(H^k\) encodes a significant amount of information about the manifold. For example, if \(H^1 \cong \mathbb{R}^1\) (as is the case for \(R^2 - \{(0, 0)\}\)), then the manifold possesses one ``hole.''

It is these kinds of questions that led Henri Poincaré to look for conditions under which a differential form is exact on \(\mathbb{R}^n\). Poincaré's early investigations led to the development of the Poincaré lemma in 1887, which established that a \(k\)-form on \(\mathbb{R}^n\) is exact if and only if it is closed. In 1889, Vito Volterra provided a complete proof of this lemma for all values of \(k\).

Poincaré's pioneering work extended to algebraic topology in 1895, where he introduced the concept of homology, defining cycles as compact submanifolds without boundaries. In 1931, Georges de Rham's doctoral thesis established a duality between de Rham cohomology and singular homology with real coefficients. Although he didn't explicitly define de Rham cohomology at the time, his work laid the foundation for its formal definition in 1938.

It turns out that whether every closed form is exact depends very much on the topology of the manifold, as we have seen in the case of \(\mathbb{R}^2\) and \(\mathbb{R}^2 - \{(0,0)\}\). De Rham cohomology, possibly the most important diffeomorphism invariant of a manifold, precisely measures the degree to which closed forms deviate from exactness.

\subsection{De Rham Cohomology}

The following definitions are due to \cite{bott1982differential}, Chapter 1.
Let $x_{1}, \ldots, x_{n}$ be the standard coordinates on $\mathbb{R}^{n}$ and define an algebra $\Omega^{*}$ over $\mathbb{R}$ generated by $d x_{1}, \ldots, d x_{n}$ subject to the relations
\[
\begin{cases}
    \; d x_{i} d x_{j}=-d x_{j} d x_{i} \qquad  &\text{for all } i \neq j \\
    \; (d x_{i})^2 = 0 \qquad &\text{for all } i
\end{cases}
\]
The smooth differential forms on $\mathbb{R}^{n}$ are elements of the algebra
$$
\Omega^{*}\left(\mathbb{R}^{n}\right):=\left\{\text {smooth functions on } \mathbb{R}^{n}\right\} \otimes_{\mathbb{R}} \Omega^{*}
$$
The wedge product of two differential forms $\tau=\sum f_{I} d x_{I}$ and $\omega=\sum g_{J} d x_{J}$ is defined as
$$
\tau \wedge \omega=\sum f_{I} g_{J} d x_{I} d x_{J}
$$
where $I$ and $J$ are a multi-indices.
The differential operator $$d: \Omega^{k}\left(\mathbb{R}^{n}\right) \rightarrow \Omega^{k+1}\left(\mathbb{R}^{n}\right)$$ 
is called exterior differentiation, and is defined by 
\[
\begin{cases}
    \; d f=\sum\left(\partial f / \partial x_{i}\right) d x_{i} \qquad &\text{if } f \in \Omega^{0}\left(\mathbb{R}^{n}\right) \\ 
    \; d \omega=\sum d f_{I} d x_{I} \qquad &\text{if } \omega=\sum f_{I} d x_{I}
\end{cases}
\]
It can be shown that $d$ is an antiderivation (\cite{bott1982differential}, Proposition 1.3) and $d^2 = 0$ (\cite{bott1982differential}, Proposition 1.4). Furthermore, it is independent of the coordinate system on $\mathbb{R}^n$.

Let $x_{1}, \ldots, x_{m}$ and $y_{1}, \ldots, y_{n}$ be the standard coordinates on $\mathbb{R}^{m}$ and $\mathbb{R}^{n}$  respectively. A smooth map $f: \mathbb{R}^{m} \rightarrow \mathbb{R}^{n}$ induces a pullback map on forms $f^{*}: \Omega^{*}\left(\mathbb{R}^{n}\right) \rightarrow \Omega^{*}\left(\mathbb{R}^{m}\right)$ in such a way that it commutes with $d$:
$$
f^{*}\left(\sum g_{I} d y_{i_{1}} \ldots d y_{i_{q}}\right)=\sum\left(g_{I} \circ f\right) d f_{i_{1}} \ldots d f_{i_{q}} \qquad \text{where } f_{i}=y_{i} \circ f
$$
Just as in the case of $\mathbb{R}^{n}$, when we have a smooth map of differentiable manifolds $f: M \rightarrow N$, it naturally induces a pullback map on forms $f^{*}: \Omega^{*}(N) \rightarrow \Omega^{*}(M)$.

We can extend the concepts of exterior derivatives and wedge products to differential forms on a manifold.
A differential form $\omega$ on a smooth manifold $M$ consists of a collection of forms $\omega_U$ for each $U$ in the atlas that defines $M$, such that if $i_u$ and $i_v$ denote the inclusions
\[\begin{tikzcd}
	& {U \cap V} \\
	U && V
	\arrow["i_v"', from=1-2, to=2-3]
	\arrow["i_u", from=1-2, to=2-1]
\end{tikzcd}\]
then the pullbacks $i_u^{*} \omega_{U}=i_v^{*} \omega_{V}$ in $\Omega^{*}(U \cap V)$. 

\vskip 0.1in

\noindent Now that we've covered some key definitions, let us begin exploring de Rham cohomology.
Let 
\begin{enumerate}
    \item[] \(M\) be a smooth manifold of dimension $n$,
    \item[] $\Omega^k(M)$ be the space of differential forms on $M$ of degree $k$, and 
    \item[] $d \colon \Omega^k(M) \to \Omega^{k+1}(M)$ denote exterior differentiation.
\end{enumerate}
Then
$$
0 \rightarrow \Omega^{0}(M) \xrightarrow[]{d} \Omega^{1}(M) \xrightarrow[]{d} \Omega^{2}(M) \xrightarrow[]{d} \cdots
$$
is a chain complex, called the \textbf{de Rham complex} of $M$.
A $k$-form $\omega \in \Omega^k(M)$ is \textbf{closed} if $\mathrm{d}\omega = 0$ and is \textbf{exact} if $\omega = \mathrm{d}\tau$ for some $\tau \in \Omega^{k-1}(M)$ (\cite{guillemin2019differential}, Definition 2.4.15).
We define
\begin{center}
    \(Z^k(M)\):= $\{ \omega \in \Omega^k(M) \mid d\omega = 0 \}$ (the vector space of all \textit{closed} \(k\)-forms)
    \vskip 0.2cm
    \(B^k(M)\):= $d(\Omega^{k-1}(M)) \subset \Omega^k(M)$ (the vector space of all \textit{exact} \(k\)-forms)
\end{center}
For all $\omega \in \Omega^k(M)$, we have $\mathrm{d} \circ \mathrm{d} (\omega) = 0$.
Hence every exact form is closed, and so \(B^k(M)\) is a subspace of \(Z^k(M)\).
The \textbf{k\textsuperscript{th} de Rham cohomology group of the manifold $M$} is the quotient vector space 
$$H^k(M) := Z^k(M)/B^k(M)$$
The construction of the quotient vector space induces an equivalence relation, where two closed forms \(\omega\) and \(\omega'\) are considered equivalent in \(Z^k(M)\) if and only if their difference is an exact form:
\begin{center}
    \(\omega' \sim \omega\) in \(Z^k(M)\) if and only if \(\omega' - \omega \in B^k(M)\)
\end{center}
The equivalence class of a closed form \([\omega]\) is called its \textbf{cohomology class}. 
The vector spaces \(H^k(M)\) also have \textit{compactly supported} counterparts, where
\begin{center}
    $Z^k_c(M) := \{ \omega \in \Omega^k_c(M) \mid d\omega = 0 \}$
    \vskip 0.2cm
    $B^k_c(M) := d(\Omega^{k-1}_c(M)) \subset \Omega^k_c(M)$
\end{center}
Then, similar to the previous case, \(B^k_c(M)\) is a  subspace of \(Z^k_c(M)\) and
\[
H^k_c(M) := Z^k_c(M)/B^k_c(M)
\]
is the \(k\)-th \textbf{compactly supported} de Rham cohomology group of \(X\).
This definition will matter when we later define Poincaré duality.

Before we move on to some computations, let us briefly discuss the ring structure on de Rham cohomology (\cite{tu2011manifolds}, 24.4). 
The vector space $\Omega^*(M)$ of smooth differential forms on $M$ is the direct sum 
$$\Omega^*(M) = \bigoplus_{k=0}^n \Omega^k(M)$$
Taking multiplication as the wedge product, this becomes a graded algebra, where the grading is the degree of the differential forms.
This product structure induces a product structure in cohomology
$$H^*(M) = \bigoplus_{k=0}^n H^k(M)$$
An element $\alpha \in H^*(M)$ is a unique, finite sum of cohomology classes $\alpha_i \in H^k(M)$ 
$$\alpha = \sum_{i=0}^n \alpha_i$$
One can check that $H^*(M)$ satisfies all the properties of a ring, where multiplication is the wedge product.
This is called the cohomology ring of $M$, and it has a natural grading by the degree of a closed form.

\subsection{Some Basic Properties}

De Rham cohomology is a powerful diffeomorphism invariant of smooth manifolds (\cite{tu2011manifolds}, Section 24.4).
However, direct computation of the de Rham groups is generally not easy. 
We begin with two results that can help narrow down our computations. The first characterizes the 0\textsuperscript{th} cohomology group, and the second gives us a bound on the dimensions we need to consider.

\begin{lemma}{Proposition \ref*{section:derham}.1: Counting Connected Components (\cite{tu2011manifolds}, 24.1)}
    If a smooth manifold \(M\) has \(r\) connected components, then \(H^0(M) = \mathbb{R}^r\).
\end{lemma}

\begin{proof}
    Suppose \(f\) is a closed 0-form on \(M\), meaning \(f\) is a smooth function on \(M\) and \(df = 0\).
    In any chart \((U, x^1, \ldots, x^n)\), we can represent \(df\) as:
    \[df = \sum_{i} \frac{\partial f}{\partial x^i} dx^i\]    
    Thus, \(df = 0\) on \(U\) if and only if all the partial derivatives \(\frac{\partial f}{\partial x^i}\) vanish entirely on \(U\), which is equivalent to \(f\) being locally constant on \(U\).
    Consequently, closed 0-forms on \(M\) are locally constant functions, meaning they remain constant within each connected component of \(M\).
    Given that \(M\) has \(r\) connected components, $df = 0$ is satisﬁed if and only if $f$ is constant on each connected component, hence it is speciﬁed by $r$ real numbers.
    Since there are no nonzero exact 0-forms, we have 
    \[H^0(M) = Z^0(M) = \{\text{closed 0-forms}\}\]
    Therefore, \(H^0(M) = \mathbb{R}^r\).
\end{proof}
\begin{corollary}{Corollary \ref*{section:derham}.1.1}
    If \(M\) is connected, then \(H^0(M) = \mathbb{R}\).
\end{corollary}
So far, we have that $H^0(M)$ is the vector space of real-valued locally constant functions on $M$, and $\dim H^0(M)$ tells us the number of connected components of $M$. What about higher dimensions?
\begin{lemma}{Proposition \ref*{section:derham}.2 \cite{tu2011manifolds}}
    If \(M\) is a smooth manifold of dimension \(n\), then \(H^k(M) = 0\) for \(k > n\).
\end{lemma}
\begin{proof}
    It suffices to show that for \(k > n\), the only \(k\)-form on \(M\) is the zero form.
    For any point \(p \in M\), the tangent space \(T_p M\) is a vector space of dimension \(n\). If \(\omega\) is a \(k\)-form on \(M\), then \(\omega_p \in A^k (T_p M)\), the space of alternating \(k\)-linear functions on \(T_p M\). 
    If $k > dim(V)$, then $A_k(V) = 0$ (\cite{tu2011manifolds}, 
Corollary 3.31). 
    Hence, for \(k > n\), we have \(A^k (T_p M) = 0\).
\end{proof}

\subsubsection{Example}
Let us compute the de Rham cohomology of the real line $\mathbb{R}^1$ (\cite{tu2011manifolds}, Example 24.3).
Since the real line is connected, by Corollary \ref*{section:derham}.1.1
$$H^0(\mathbb{R}^1) = \mathbb{R}$$
There are no nonzero 2-forms on $\mathbb{R}^1$, which implies every 1-form on $\mathbb{R}^1$ is closed. 
A 1-form $f(x) \, dx$ on $\mathbb{R}^1$ is exact if and only if there is a smooth function $g(x)$ on $\mathbb{R}^1$ such that $f(x) \, dx = dg = g'(x) \, dx$. Such a function $g(x)$ is simply an antiderivative of $f(x)$, therefore every 1-form on $\mathbb{R}^1$ is exact. 
Thus, $$H^1(\mathbb{R}^1) = 0$$
Applying Proposition \ref*{section:derham}.2, we have $$H^k(\mathbb{R}^1) = \begin{cases} \mathbb{R} & \text{for } k = 0, \\ 0 & \text{for } k \geq 1. \end{cases}$$

\section{Major Tools of de Rham Cohomology}
\label{section:tools}

To give the reader a general idea of how computations are done, we present four fundamental tools for computing de Rham cohomology and demonstrate a few examples. We defer the proofs of these tools until after we've seen some interesting applications.

\subsection{Mayer-Vietoris}
Our primary tool is the Mayer-Vietoris theorem, which provides a method for computing the cohomology of the union of two open sets. These techniques fall under the notion of ``diagram chasing'' and mastering them is more akin to becoming proficient in checkers or crossword puzzles than the mathematics to which they are applied. However, despite their seemingly mechanical nature, these techniques offer a powerful method for systematically breaking down the cohomology of complex spaces into manageable pieces.

The Mayer–Vietoris theorem addresses the following question: If we have information about the cohomology groups of two out of three complexes, what can we infer about the cohomology groups of the third?
Specifically, it allows us to calculate $H^k(U \cup V)$ as a ``function'' of $H^k(U)$, $H^k(V)$ and $H^k(U \cap V)$. By iteration, we get a computation for $H^k(U_1 \, \cup \, U_2 \, \cup \dots \, \cup \, U_n)$. Combined with the Poincaré lemma, this yields a principal calculation suitable for many general cases.

\begin{theorem}{Theorem \ref*{section:tools}.1: Mayer-Vietoris theorem (\cite{tu2011manifolds}, 26.1)}
Let $M = U \cup V$ be a smooth manifold where $U$ and $V$ are open.
Denote the inclusion maps
\[\begin{tikzcd}
	{U \cap V} & U \\
	V & M
	\arrow["{j_u}", hook, from=1-1, to=1-2]
	\arrow["{j_v}"', hook, from=1-1, to=2-1]
	\arrow["{i_u}", hook, from=1-2, to=2-2]
	\arrow["{i_v}"', hook, from=2-1, to=2-2]
\end{tikzcd}\]
which induce pullback maps on differential forms,
\[\begin{tikzcd}
	\Omega^k({U \cap V}) & \Omega^k(U) \\
	\Omega^k(V) & \Omega^k(M)
	\arrow["{j^*_u}", hook, from=1-2, to=1-1]
	\arrow["{j^*_v}"', hook, from=2-1, to=1-1]
	\arrow["{i^*_u}", hook, from=2-2, to=1-2]
	\arrow["{i^*_v}"', hook, from=2-2, to=2-1]
\end{tikzcd}\]

\vskip 0.1in

The following sequence of cochain complexes is short exact
\[\begin{tikzcd}
	0 & {\Omega^*(M)} & {\Omega^*(U) \oplus \Omega^*(V)} & {\Omega^*(U \cap V)} & 0
	\arrow[from=1-1, to=1-2]
	\arrow["i", from=1-2, to=1-3]
	\arrow["j", from=1-3, to=1-4]
	\arrow[from=1-4, to=1-5]
\end{tikzcd}\]
where
\begin{align*}
    i(\omega) &= (i^*_u \oplus i^*_v)\omega = (i^*_u \omega, i^*_v \omega) = (\omega|_U, \omega|_V) \\
    j(\omega, \eta) &= (j^*_u - j^*_v)(\omega, \eta) = j^*_u(\omega) - j^*_v(\eta) = \omega|_{U \cap V} - \eta|_{U \cap V}
\end{align*}
and gives rise to a long exact sequence in cohomology:
\begin{center}
    \begin{tikzpicture}[descr/.style={fill=Purple!10!white,inner sep=1.5pt}]
        \matrix (m) [
            matrix of math nodes,
            row sep=2em,
            column sep=2.5em,
            text height=1.5ex, text depth=0.25ex
        ]     
        { 0 & H^0(M) & H^0(U) \oplus H^0(V) & H^0(U \cap V) \\
            & H^1(M) & H^1(U) \oplus H^1(V) & H^1(U \cap V) \\
            & H^2(M) & H^2(U) \oplus H^2(V) & H^2(U \cap V) \\
            & \mbox{}         &                 & \mbox{}         \\
            & H^k(M) & H^k(U) \oplus H^k(V) & H^k(U \cap V) & \cdots \\
        };

        \path[overlay,->, font=\scriptsize,>=latex]
        (m-1-1) edge (m-1-2)
        (m-1-2) edge node[descr,yshift=1.2ex] {$i^*$} (m-1-3)
        (m-1-3) edge node[descr,yshift=1.2ex] {$j^*$} (m-1-4)
        (m-1-4) edge[out=355,in=175,red] node[descr,yshift=0.3ex] {$\delta$} (m-2-2)
        (m-2-2) edge (m-2-3)
        (m-2-3) edge (m-2-4)
        (m-2-4) edge[out=355,in=175,red] node[descr,yshift=0.3ex] {$\delta$} (m-3-2)
        (m-3-2) edge (m-3-3)
        (m-3-3) edge (m-3-4)
        (m-3-4) edge[out=355,in=175,dashed,red] node[descr,yshift=0.2ex,xshift=-0.6ex] {$\quad \dots \;$} (m-5-2)
        (m-5-2) edge node[descr,yshift=1.2ex] {$i^*$} (m-5-3)
        (m-5-3) edge node[descr,yshift=1.2ex] {$j^*$} (m-5-4)
        (m-5-4) edge [red] (m-5-5);
    \end{tikzpicture}
\end{center}

where $i^*$ and $j^*$ are induced from $i$ and $j$ as follows:
\begin{align*}
i^*[\omega] &= [i(\omega)] = ([\omega|_U], [\omega|_V]) \in H^k(U) \oplus H^k(V), \\
j^*([\omega], [\eta]) &= [j(\omega, \eta)] = [\omega|_{U \cap V}] - [\eta|_{U \cap V}] \in H^k(U \cap V)
\end{align*}
and $\delta: H^k(U \cap V) \rightarrow H^{k+1}(U \cup V)$ is a connecting homomorphism.
\end{theorem}

In defining the connecting homomorphism, one must choose a cocycle $c$ to represent the cohomology class $[c] \in H^k(U \cap V)$ and an element that maps to $c$ under $j$.
For a precise formulation of $\delta$, see \cite{tu2011manifolds}, page 291.

\subsubsection{Example}

We compute the de Rham cohomology of the circle $S^1$ (\cite{tu2011manifolds}, 26.2).
Consider an open cover of $S^1$ consisting of two open arcs $U$ and $V$. 
\begin{center}
    \includegraphics[width=0.4\textwidth]{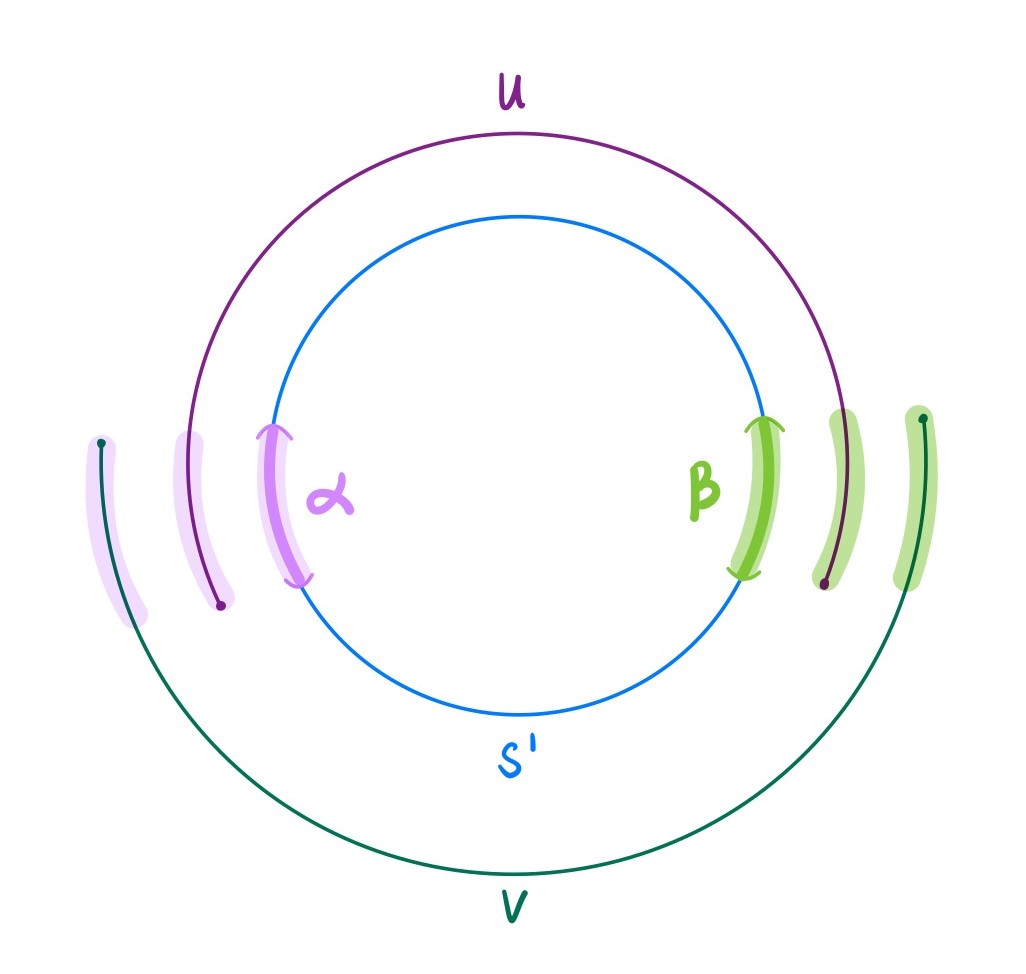}
\end{center}
Since an open arc is diffeomorphic to $\mathbb{R}^1$, the cohomology rings $H^*(U)$ and $H^*(V)$ are isomorphic to $H^*(\mathbb{R}^1)$.
Now, their intersection $U \cap V$ is the disjoint union of open arcs $\alpha$ and $\beta$. 
Thus, $H^*(U \cap V)$ is isomorphic to that of the disjoint union $H^*(\mathbb{R}^1 \coprod \mathbb{R}^1)$. 
We obtain the following Mayer-Vietoris sequence:

\vskip 0.2in

\noindent \adjustbox{scale=0.8,center}{
\begin{tikzcd}
	& {} &&&& {} \\
	& 0 & 0 & 0 && {H^2(S^1)} & {H^2(U) \oplus H^2(V)} & {H^2(U \cap V)} \\
	& {H^1(S^1)} & 0 & 0 && {H^1(S^1)} & {H^1(U) \oplus H^1(V)} & {H^1(U \cap V)} \\
	0 & {\mathbb{R}} & {\mathbb{R} \oplus \mathbb{R}} & {\mathbb{R} \oplus \mathbb{R}} & 0 & {H^0(S^1)} & {H^0(U) \oplus H^0(V)} & {H^0(U \cap V)}
	\arrow[from=4-1, to=4-2]
	\arrow[from=4-2, to=4-3]
	\arrow[from=4-3, to=4-4]
	\arrow[from=4-5, to=4-6]
	\arrow[from=4-6, to=4-7]
	\arrow[from=4-7, to=4-8]
	\arrow[from=4-8, to=3-6]
	\arrow[from=3-6, to=3-7]
	\arrow[from=3-7, to=3-8]
	\arrow[from=3-8, to=2-6]
	\arrow[from=2-6, to=2-7]
	\arrow[from=2-7, to=2-8]
	\arrow["\cdots"{description}, from=2-8, to=1-6]
	\arrow[from=4-4, to=3-2]
	\arrow[from=3-2, to=3-3]
	\arrow[from=3-3, to=3-4]
	\arrow[from=3-4, to=2-2]
	\arrow[from=2-2, to=2-3]
	\arrow[from=2-3, to=2-4]
	\arrow["\cdots"{description}, from=2-4, to=1-2]
\end{tikzcd}
}

\vskip 0.2in

\noindent We then extract the following exact sequence:
\[\begin{tikzcd}
	0 & {\mathbb{R}} & {\mathbb{R} \oplus \mathbb{R}} & {\mathbb{R} \oplus \mathbb{R}} & {H^1(S^1)} & 0
	\arrow[from=1-1, to=1-2]
	\arrow["{i^*}", from=1-2, to=1-3]
	\arrow["{j^*}", from=1-3, to=1-4]
	\arrow["{d^*}", from=1-4, to=1-5]
	\arrow[from=1-5, to=1-6]
\end{tikzcd}\]
By the Rank Nullity theorem $\sum_{k=0}^{3} (-1)^k \dim H^k = 0$, so 
\begin{align*}
    &(-1)^0 \dim (\mathbb{R}) + (-1)^1 \dim (\mathbb{R} \oplus \mathbb{R}) +(-1)^2 \dim (\mathbb{R} \oplus \mathbb{R}) + (-1)^3 \dim (H^1(S^1)) = 0 \\
    &\implies 1 - 2 + 2 - \dim H^1(S^1) = 0 \\
    &\implies \dim H^1(S^1) = 1
\end{align*}
We conclude that $\dim H^1(S^1) \cong \mathbb{R}$. Hence, the cohomology of the circle is given by:
\[
H^k(S^1) = \begin{cases}
\mathbb{R} & \text{for } k = 0, 1 \\
0 & \text{otherwise}
\end{cases}
\]

\subsection{Homotopy Invariance}
The next tool is simple to state, but it is a very useful result. Since we are primarily interested in smooth manifolds, our notion of homotopy will be smooth homotopy. It deviates from the typical topological notion of homotopy only in that we require all the involved maps to be smooth.

\begin{theorem}{Theorem \ref*{section:tools}.2: Homotopy Axiom (\cite{tu2011manifolds}, 29.1)}
    Two smoothly homotopic maps $f, g: M \to N$ of manifolds induce the same map in cohomology:
    \[ f^* = g^* : H^k(N) \to H^k(M) \]
\end{theorem}
This allows us to identify two important results.
Firstly, homotopy equivalence induces an isomorphism in cohomology:
\begin{corollary}{Corollary \ref*{section:tools}.2.1: Homotopy Equivalence (\cite{tu2011manifolds}, 27.4)}
    If $f: M \to N$ is a homotopy equivalence, then the induced map in cohomology
    \[ f^* : H^*(N) \to H^*(M) \]
    is an isomorphism.
\end{corollary}
In turn, this allows us to characterize the cohomology of contractible spaces. 
\begin{corollary}{Corollary \ref*{section:tools}.2.2: Poincaré Lemma (\cite{tu2011manifolds}, 27.13)}
    Since $\mathbb{R}^n$ has the homotopy type of a point, the cohomology of $\mathbb{R}^n$ is:
    \[ H^k(\mathbb{R}^n) = \begin{cases} \mathbb{R}, & \text{for } k = 0, \\ 0, & \text{for } k > 0 \end{cases} \]
\end{corollary}
In other words, any contractible manifold will have the same cohomology as a
point.

\subsubsection{Example}

We consider the de Rham cohomology of $M = \mathbb{R}^2 - \{p\} $(\cite{tu2011manifolds}, Example 27.14).
For any point \(p\) in the plane \(\mathbb{R}^2\), consider the translation \(x \mapsto x - p\).
This translation is a diffeomorphism of the plane with point \(p\) removed, denoted \(\mathbb{R}^2 - \{p\}\), with the punctured plane \(\mathbb{R}^2 - \{(0, 0)\}\). 
The punctured plane and the circle \(S^1\) have the same homotopy type, consequently, they have isomorphic cohomology. Thus, by our previous example,
\[
H^k(M) = H^k(S^1) = \begin{cases}
\mathbb{R} & \text{for } k = 0, 1 \\
0 & \text{otherwise}
\end{cases}
\]

\subsection{Poincaré Duality}
Previously, we've defined two ``types'' of cohomology groups: the de Rham cohomology groups \(H^k\) and the compactly supported de Rham cohomology groups \(H_c^{k}\). Poincaré duality tells us these are, in fact, strongly related: if \(M\) is a connected oriented \(n\)-manifold, then \(H_c^{(n-k)}(M)\) is the vector space dual of \(H^k(M)\).

\begin{theorem}{Theorem \ref*{section:tools}.3: Poincaré Duality (\cite{madsen1997calculus}, 13)}
    Given an oriented manifold $M$ of dimension $n$,
    \[ H^k(M) \cong H_c^{n-k}(M)^*, \quad \text{for } k \in \mathbb{Z}, \]
    where $H_c^{n-k}(M)^*$ denotes the dual vector space of linear forms on $H^{n-k}(M)$ with compact support.
\end{theorem}
Note that the reverse $H_c^k(M) \cong H^{n-k}(M)^*$ is not always true. The lack of symmetry in this context stems from the fact that while the dual of a direct sum is a direct product, the dual of a direct product is not a direct sum.
It follows that Poincaré Duality can be used to immediately characterize connected, oriented manifolds.
\begin{corollary}{Corollary \ref*{section:tools}.3.1: Connected Manifolds (\cite{bott1982differential}, 5.8)}
    If $M$ is a connected oriented manifold of dimension $n$, then
    \[ H_c^n(M) \cong \mathbb{R} \]
    In particular, if $M$ is \textit{compact}, oriented, and connected,
    \[ H^n(M) \cong \mathbb{R} \]
\end{corollary}

\subsection{The Künneth Formula}
The last tool we introduce, the Künneth formula, allows us to express the cohomology of the direct product of spaces in terms of the cohomology of its factors.

\begin{theorem}{Theorem \ref*{section:tools}.4: Künneth Formula (\cite{bott1982differential}, 5.9)}
    The cohomology of the product of two smooth manifolds \(M\) and \(N\) is the tensor product:
    \[H^*(M \times N) = H^*(M) \otimes H^*(N).\]
    This means that 
    \[H^k(M \times N) = \bigoplus_{p+q=k} H^p(M) \otimes H^q(N)\]
    for every nonnegative integer \(k\)
\end{theorem}

\subsubsection{Example}

Consider the torus $T^2 = S^1 \times S^1$. 
Since $H^0(S^1) = 0$ and $H^1(S^1) = \mathbb{R}$, we have
\begin{equation*}
H^0(T^2) = \mathbb{R} \otimes \mathbb{R} = \mathbb{R}
\end{equation*}
\begin{equation*}
H^1(T^2) = (\mathbb{R} \otimes \mathbb{R}) \oplus (\mathbb{R} \otimes \mathbb{R}) = \mathbb{R} \oplus \mathbb{R}
\end{equation*}
\begin{equation*}
H^2(T^2) = H^0(T^2) = \mathbb{R}
\end{equation*}
Note the Poincaré duality $H^0(T^2) = H^2(T^2)$.

\section{Repaying Technical Debt}

With the four tools we have developed, we can now compute the de Rham cohomology of many manifolds.
It's time to repay our technical debt and prove the theorems we have stated without proof.

\subsection{Mayer-Vietoris}
In this section, we present a proof of the Mayer–Vietoris theorem. 
The proof can be divided into two parts. The first part is an algebraic result known as the ``Zigzag Lemma'', the proof of which we omit but can be found in \cite{madsen1997calculus} (Chapter 4). The second part is a differential geometric part which demonstrates that we satisfy the necessary conditions of the Zigzag Lemma.

\subsubsection{The Tools}

\begin{lemma}{Lemma 4.1: The Zigzag Lemma (\cite{lee2012smooth}, 17.40)}
    Given a short exact sequence of chain complexes 
    \[
    0 \rightarrow A^* \xrightarrow[]{f} B^* \xrightarrow[]{g} C^* \rightarrow 0
    \]
    for each positive integer $k$ there is a linear map 
    \[
    \delta: H^k(C^*) \rightarrow  H^{k+1}(A^*) 
    \]
    such that the following sequence is exact:
    \[
    \cdots \xrightarrow[]{\delta} H^k(A^*) \xrightarrow[]{f^*} H^k(B^*) \xrightarrow[]{g^*} H^k(C^*)  \xrightarrow[]{\delta} H^{k+1}(A^*) \xrightarrow[]{f^*} \cdots
    \]
\end{lemma}

\begin{proof}
    See \cite{madsen1997calculus}, Theorem 4.9.
\end{proof}

\subsubsection{The Proof}

The following proof is based on the one provided in \cite{lee2012smooth} (Chapter 17), and involves demonstrating the exactness of the de Rham complex for each $k$. Since the pullback maps and the exterior derivative commute (\cite{lee2012smooth}, Proposition 14.26), this establishes a short exact sequence of cochain maps, and the Mayer–Vietoris theorem can be readily derived from the Zigzag Lemma.

\begin{proof}

    Suppose $M$ is a smooth manifold with or without boundary, and $U, V$ are open subsets such that $M = U \cup V$. 
    Let the maps $i$ and $j$ be defined as in Theorem \ref*{section:tools}.1.
    It suffices to show that for each integer $k \geq 0$, the de Rham complex
    \[0 \rightarrow \Omega^k(M) \xrightarrow{i} \Omega^k(U) \oplus \Omega^k(V) \xrightarrow{j} \Omega^k(U \cap V) \rightarrow 0\] 
    is exact.

    \subsubsection{Part 1}
    \textbf{Exactness at $\Omega^k(M)$}

    \noindent It suffices to show that $i = i_u^* \oplus i_v^*$ is injective. 
    Suppose $\omega \in \Omega^k(M)$ such that
    \begin{align*}
        i(\omega) &= (i_u^* \oplus i_v^*)\omega \\
        &= (i_u^*(\omega), i_v^*(\omega)) \\
        &= (\omega_U, \omega_V) \\
        &= (0,0)
    \end{align*}
    Since the restrictions of $\omega$ to $U$ and $V$ are both zero and $\{U, V\}$ is an open cover of $M$, this implies that $\omega$ is zero.

    \subsubsection{Part 2}
    \textbf{Exactness at $\Omega^k(U) \oplus \Omega^k(V)$}

    \noindent It suffices to show that $\text{Ker}(j_u^* - j_v^*) = \text{Im}(i_u^* \oplus i_v^*)$.
    We have
    \begin{align*}
        (j_u^* - j_v^*) \circ (i_u^* \oplus i_v^*)(\omega) &= (j_u^* - j_v^*)(\omega _U, \omega_V) \\
        &= \omega_{U \cap V} - \omega_{U \cap V} \\
        &= 0
    \end{align*}
    Thus, $\text{Im}(i_u^* \oplus i_v^*) \subseteq \text{Ker}(j_u^* - j_v^*)$.
    Conversely, suppose $(j_u^* - j_v^*)(\omega, \eta) = 0$ for some $(\omega, \eta) \in \Omega^k(U) \oplus \Omega^k(V)$.
    We have
    \begin{align*}
        &(j_u^* - j_v^*)(\omega, \eta) = 0 \\
        &\implies j_u^* \omega = j_v^* \eta \\
        &\implies \omega_{U \cap V} = \eta_{U \cap V}
    \end{align*}
    This implies there is a $k$-form $\sigma$ on $M$ such that
    \begin{align*}
        \sigma = \begin{cases}
            \omega \quad \text{on U} \\
            \eta \quad \text{on V}
        \end{cases}
    \end{align*}
    Since $(i_u^* \oplus i_v^*)\sigma = (\sigma|_U, \sigma|_V) = (\omega, \eta)$, we have $\text{Ker}(j_u^* - j_v^*) \subseteq \text{Im}(i_u^* \oplus i_v^*)$.

    \subsubsection{Part 3}
    \textbf{Exactness at $\Omega^k(U \cap V)$}

    \noindent It suffices to show $i^* = j^*$ is surjective.
    Let $\sigma \in \Omega^k(U \cap V)$.
    We show there exist $\omega \in \Omega^k(U)$ and $\eta \in \Omega^k(V)$ such that 
    \begin{align*}
        \sigma &= (j_u^* - j_v^*)(\omega, \eta) \\
        &= j_u^* \omega - j_v^* \eta \\
        &= \omega|_{U \cap V} - \eta|_{U \cap V}
    \end{align*}

    Let $f, g$ be a smooth partition of unity subordinate to the open cover $\{U, V\}$ of $M$. 
    Define $\theta \in \Omega^k(U)$ by
    \[
    \theta =
    \begin{cases}
    g \omega& \text{on } U \cap V \\
    0 & \text{on } U \setminus \text{supp } g
    \end{cases}
    \]
    Both $g \omega$ and $0$ are zero on $\{ (U \cap V) \setminus \text{supp }g \}$, so $\theta$ is a smooth $k$-form on $U$. Similarly, define a smooth $k$-form on $V$ by
    \[
    \theta' =
    \begin{cases}
    -f \omega& \text{on } U \cap V \\
    0 & \text{on } V \setminus \text{supp } f
    \end{cases}
    \]
    Then, since $f + g = 1$, we have 
    \begin{align*}
        \theta|_{U \cap V} - \theta'|_{U \cap V} &= g \omega - (- f \omega) \\
        &= (f + g) \omega \\
        &= \omega
    \end{align*}
    
\end{proof}


Admittedly, in the interest of time and space, we do not include full proofs of the remaining theorems. 
However, we give an idea of what they entail and provide corresponding references.

\subsection{Homotopy Invariance}
It suffices to show that homotopies of continuous maps do not affect the induced map on the cohomology groups. This can be done by reducing the problem to two special maps which are the 0-section and the 1-section of a product line bundle. Then, by using a suitable cochain homotopy of these maps, one can prove that they induce the same map in cohomology.

\begin{proof} \textit{of Homotopy Invariance}

    See \cite{tu2011manifolds}, Chapter 29.
\end{proof}

\subsection{Poincaré Duality}
When dealing with a connected, oriented manifold $M$ of dimension $n$ which has a finite topology, it can be shown that the vector spaces $H_c^{n-k}(M)$ and $H^k(M)$ are both finite-dimensional. One can establish a natural bilinear pairing between these spaces, which in turn leads to a natural linear mapping of $H^k(M)$ to the vector space dual of $H_c^{n-k}(M)$.

\begin{proof} \textit{of Poincaré duality}

    See \cite{guillemin2019differential}, Section 5.4; or \cite{madsen1997calculus} Chapter 13; or \cite{bott1982differential} Chapter 5 (p.44)
\end{proof}

\subsection{The Künneth Formula}
The method proceeds by induction, assuming that the manifolds $M$ and $F$ have finite good covers.
Using natural projections $\pi$ and $\rho$ from $M \times F$ to $M$ and $F$ respectively, we obtain a map on forms that takes the tensor product of two forms $\omega \otimes \phi$ into $\pi^*\omega \wedge \rho^*\phi$.
This induces a map
\[
\psi : H^*(M) \otimes H^*(F) \rightarrow H^*(M \times F).
\]
and one can show this map is an isomorphism.

\begin{proof} \textit{of Künneth Formula}

    See \cite{bott1982differential}, Chapter 5 (p.47)
\end{proof}


\section{de Rham's Theorem}
\label{section:derhamtheorem}

Up to this point, we have mainly discussed some tools for computing de Rham cohomology. 
We've also claimed that de Rham cohomology identifies global properties of a manifold, however, we have yet to give a very good reason to believe this.
Our next topic, De Rham's theorem, establishes a natural isomorphism between the de Rham cohomology of a manifold $M$ and its singular cohomology with coefficients in $\mathbb{R}$. It exemplifies a fundamental connection between the analytical aspects of a manifold, which involve solutions to differential equations of the form $d\eta = \omega$ for closed forms $\omega$, and the topological aspects, which pertain to the characterization of ``holes'' in various dimensions.
The natural isomorphism, derived from Stokes' theorem, arises through the pairing of differential forms with smooth singular chains, resulting in a mapping from the k\textsuperscript{th} de Rham cohomology group to the k\textsuperscript{th} singular cohomology group with coefficients in $\mathbb{R}$.
This exposition is based on the proof outlined in \cite{lee2012smooth} Chapter 18, which itself is based on \cite{bredon2013topology}, Chapter 9. Another approach is via the theory of sheaves; for example, see \cite{warner1983foundations}.
To avoid confusion, going forward we denote de Rham cohomology $H_\Omega^k$ and singular cohomology $H^k$.

\subsection{Smooth Singular Homology}

We will establish a connection between the singular and de Rham cohomology groups by integrating differential forms over singular chains. 
Consider a scenario where we have a singular $k$-simplex $\sigma$ in a manifold $M$ and a $k$-form $\omega$ defined on $M$.
We wish to pull $\omega$ back by $\sigma$ and integrate the resulting form over $\sigma$.
However, there's an immediate issue with this approach.
Forms can only be pulled back by smooth maps, whereas singular simplices are generally continuous, but not necessarily smooth.
To address this issue, it can be shown that singular homology can be computed equally well using only smooth simplices.

If $M$ is a smooth manifold, a smooth $k$-simplex in $M$ is a smooth map $\sigma: \Delta_k \rightarrow M$.
The smooth chain group in degree $k$ is the subgroup generated by finite formal linear combinations of smooth simplices and is denoted $C_k^\infty(M)$.
Elements of this group are called smooth chains. 
Since the boundary of a smooth simplex is a smooth chain, we can define the $k$\textsuperscript{th} smooth singular homology group of $M$ to be the quotient group
\[
H_k^\infty(M) = \frac{\text{Ker}(\partial : C_k^\infty(M) \rightarrow C_{k-1}^\infty(M))}{\text{Im}(\partial : C_{k+1}^\infty(M) \rightarrow C_k^\infty(M))}.
\]
The inclusion $\iota : C^\infty_k(M) \hookrightarrow C_k(M)$ commutes with the boundary operator in homology $\iota \circ \partial = \partial \circ \iota$, thus it induces a map on homology $\iota_* : H^\infty_k(M) \rightarrow H_k(M)$ defined by $\iota_*[c] = [\iota(c)]$. Furthermore, this map is an isomorphism.

\begin{lemma}{\ref*{section:derhamtheorem}.1 Singular and Smooth Singular Homology (\cite{lee2012smooth}, Theorem 18.7)}
    For any smooth manifold $M$, the map $\iota_*: H^\infty_k(M) \rightarrow H_k(M)$ induced by inclusion is an isomorphism.
\end{lemma}
The key to the proof is the construction of a homotopy from each continuous simplex to a smooth one that respects boundary faces.
To do so we construct two operators:
A smoothing operator $s : C_k(M) \rightarrow C^{\infty}_k(M)$ such that $s \circ \partial = \partial \circ s$ and $s \circ \iota$ is the identity on $C^{\infty}_k(M)$, and a homotopy operator $H_\sigma$ demonstrating that $\iota \circ s$ induces the identity map on $H_k(M)$.
The proof is rather long and technical and does not involve de Rham cohomology, so we omit it.
\begin{proof}
    See \cite{lee2012smooth}, Theorem 18.7.
\end{proof}
\noindent However, we will make use of this result: for the remainder of the paper, we take $[c] \in H_k(M) \cong H^\infty_k(M)$ to be smooth.

\subsection{The Tools}

Let $M$ be a smooth manifold, $\omega \in \Omega^k(M)$ be a closed $k$-form on $M$, and $\sigma : \Delta_k \rightarrow M$ be a smooth $k$-simplex in $M$.
We define the integral of $\omega$ over $\sigma$ as
\[
\int_{\sigma} \omega = \int_{\Delta_k} \sigma^*\omega,
\]
where $\Delta_k \subset \mathbb{R}^k$ is the domain of integration and $\sigma^*\omega$ is the pullback. 
If $c = \sum_{i=1}^l c_i\sigma_i \in C^\infty_k(M)$ is a smooth $k$-chain, we define the integral of $\omega$ over $c$ as
\[
\int_{c} \omega = \sum_{i=1}^l c_i \int_{\sigma_i} \omega
\]
A generalized version of Stokes' theorem for this definition is as follows.
\begin{lemma}{\ref*{section:derhamtheorem}.2 Stokes’ Theorem for Chains (\cite{lee2012smooth}, Proposition 18.12)}
    Let $M$ be a smooth manifold.
    If $c \in C^\infty_k(M)$ is a smooth $k$-chain and $\omega \in \Omega^{k-1}(M)$ is a smooth $(k-1)$-form, then
    \[
    \int_{\partial c} \omega = \int_c d\omega
    \]
\end{lemma}
\begin{proof}
    See \cite{lee2012smooth}, Proposition 18.12.
\end{proof}

Stokes' theorem allows us to define a linear map 
$$\mathscr{I}: H^k_{\Omega}(M) \rightarrow H^k(M; \mathbb{R})$$ 
called the de Rham homomorphism. 
By the universal coefficient theorem, the map $H^k(M;\mathbb{R}) \rightarrow \text{Hom}(H_k(M),\mathbb{R})$ is an isomorphism.
Thus, for any $[\omega] \in H^k_{\Omega}(M)$ and $[c] \in H_k(M)$ we may define 
$$\mathscr{I}([\omega])[c] = \int_{\bar{c}} \omega$$
where $\bar{c}$ is any smooth $k$-cycle representing the homology class $[c]$ and $\mathscr{I}([\omega])$ is the cohomology class corresponding to the element of $\text{Hom}(H_k(M), R)$ defined by $[c] \mapsto \int_{\bar{c}} \omega$.

If $\bar{c}_1, \bar{c}_2$ are smooth cycles representing the same homology class, then the equivalence of singular homology and smooth singular homology (Lemma \ref*{section:derhamtheorem}.1) guarantees that $\bar{c}_1 - \bar{c}_2 = \partial p$ for some smooth (k+1)-chain $p \in C^\infty_{k+1}(M)$. Since $\omega$ represents a cohomology class we have $d\omega = 0$, thus
\[
\int_{\bar{c}_1} \omega - \int_{\bar{c}_2} \omega = \int_{\partial p} \omega = \int_{p} d\omega = 0,
\]
If $\omega = d \eta$ is exact, since $\bar{c}$ represents a homology class we have $\partial \bar{c} = 0$, so
\[
\int_{\bar{c}} \omega = \int_{\bar{c}} d\eta = \int_{\partial \bar{c}} \eta = 0
\]

The resulting homomorphism $\mathscr{I}([\omega]): H_k(M) \rightarrow \mathbb{R}$ depends linearly on $\omega$. Therefore, $\mathscr{I}([\omega])$ is a well-defined element of $\text{Hom}(H_k(M); \mathbb{R}) \cong H^k(M; \mathbb{R})$. Furthermore, this map is natural:
\begin{lemma}{\ref*{section:derhamtheorem}.3 Naturality of the de Rham Homomorphism (\cite{lee2012smooth}, Proposition 18.13)}
   Let $\mathscr{I} : H^k_{\Omega}(M) \rightarrow H^k(M; \mathbb{R})$ denote the de Rham homomorphism.
   \begin{enumerate}
       \item[a)] For a smooth map $F: M \rightarrow N$, the following diagram commutes:
    \[\begin{tikzcd}
    	{H_{\Omega}^k(N)} && {H_{\Omega}^k(M)} \\
    	\\
    	{H^k(N;\mathbb{R})} && {H^k(M;\mathbb{R})}
    	\arrow["{F^*}", from=1-1, to=1-3]
    	\arrow["{F^*}", from=3-1, to=3-3]
    	\arrow["{\mathscr{I}}", from=1-1, to=3-1]
    	\arrow["{\mathscr{I}}", from=1-3, to=3-3]
    \end{tikzcd}\]
       \item[b)] For open subsets $U, V$ of $M$ such that $M = U \cup V$, the following diagram commutes:
       \[\begin{tikzcd}
	{H_{\Omega}^{p-1}(U \cap V)} && {H_{\Omega}^k(M)} \\
	\\
	{H^{k-1}(U \cap V;\mathbb{R})} && {H^k(M;\mathbb{R})}
	\arrow["\delta", from=1-1, to=1-3]
	\arrow["{\partial^*}", from=3-1, to=3-3]
	\arrow["{\mathscr{I}}", from=1-1, to=3-1]
	\arrow["{\mathscr{I}}", from=1-3, to=3-3]
        \end{tikzcd}\]
       where $\delta$ and $\partial^*$ are the connecting homomorphisms defined by the Mayer–Vietoris sequences for de Rham and singular cohomology, respectively.
   \end{enumerate}
    
\end{lemma}
\begin{proof}
    Consider (a); first, we wish show that $\mathscr{I} \circ F^* = F^* \circ \mathscr{I}$.
    Suppose $\sigma$ is a smooth $k$-simplex in $M$ and $\omega \in \Omega^k(N)$ is a smooth $k$-form on $N$.
    By definition,
    \[
    \int_{\sigma} F^* \omega =  \int_{\Delta_k} \sigma^* F^* \omega = \int_{\Delta_k} (F \circ \sigma)^* \omega = \int_{F \circ \sigma} \omega
    \]
    We can extend this linearly to show
    \[
    \mathscr{I}(F^*[\omega])[\sigma] = \mathscr{I}[\omega][F \circ \sigma] = \mathscr{I}[\omega] F^*([\sigma]) = F^*(\mathscr{I}[\omega])[\sigma]
    \]
    
    Now consider (b); we wish show that $\mathscr{I} \circ \delta = \partial^* \circ \mathscr{I}$.
    Since $H^k(M; \mathbb{R}) \cong \text{Hom}(H_k(M); \mathbb{R})$, we can rewrite this as
    $$\mathscr{I}(\delta[\omega])[e] = \mathscr{I}([\omega])(\partial_* [e])$$ 
    for any $[\omega] \in H^{k-1}_{\Omega}(U \cap V)$ and any $[e] \in H_k(M)$.
    This is equivalent to $$\int_e \sigma = \int_c \omega$$
    for a smooth $k$-form $\sigma$ representing $\delta[\omega]$ and a smooth $(k-1)$-chain $c$ representing $\partial_* [e]$.
    
    By the Mayer–Vietoris theorem for singular homology (\cite{lee2012smooth}, Theorem 18.4), we can let $c = \partial f$, where $f$ and $f'$ are smooth $p$-chains in $U$ and $V$, respectively, such that $f + f'$ represents the same homology class as $e$.
    \begin{center}
        \includegraphics[width=0.6\textwidth]{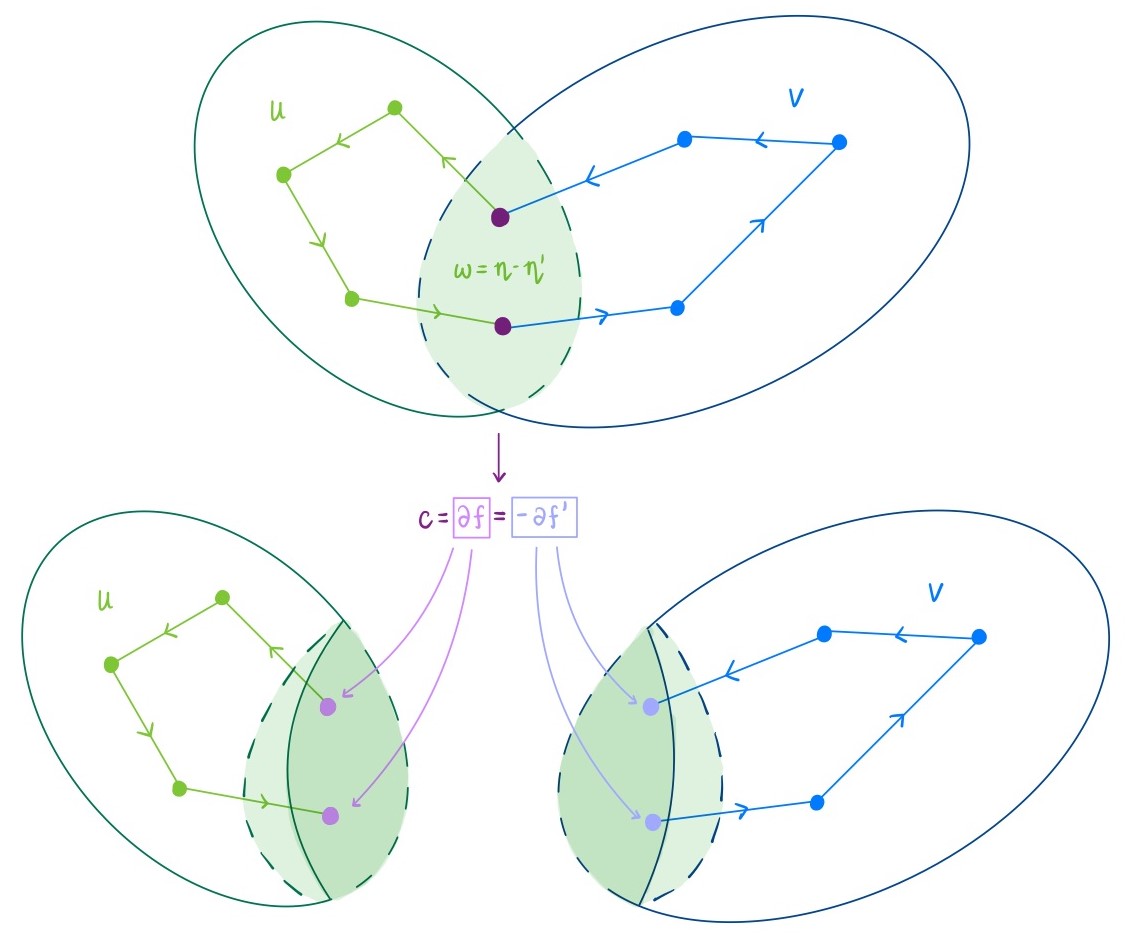}
    \end{center}
    Appealing to the well-definedness of the Mayer–Vietoris sequence, we can choose $\eta \in \Omega^{k-1}(U)$ and $\eta' \in \Omega^{k-1}(V)$ such that $\omega = \eta|_{U \cap V} - \eta'|_{U \cap V}$, and $\sigma$ is the $k$-form equal to $d\eta$ on $U$ and to $d\eta'$ on $V$. 
    
    Since $\partial f + \partial f' = \partial e = 0$, we have $\partial f = - \partial f'$.
    Similarly $d\eta|_{U \cap V} - d\eta'|_{U \cap V} = d\omega = 0$
    Therefore, we have
    \begin{align*}
        \int_c \omega &= \int_{\partial f} \omega \\
        &= \int_{\partial f} \eta - \int_{\partial f} \eta' \\
        &= \int_{\partial f} \eta + \int_{\partial f'} \eta' \qquad \text{since } \partial f = - \partial f'\\
        &= \int_{f} d \eta + \int_{f'} d \eta' \qquad \text{by Stokes’ Theorem for chains} \\
        &= \int_f \sigma + \int_{f'} \sigma \\
        &= \int_e \sigma
    \end{align*}
\end{proof}

\subsection{The Theorem}

We may now turn our attention to stating and proving the fundamental and beautiful result that is de Rham's Theorem.

\begin{theorem}{\ref*{section:derhamtheorem}.4: de Rham's Theorem \cite{lee2012smooth}}
    For every smooth manifold $M$ and nonnegative integer $p$, the de Rham homomorphism $$\mathscr{I} : H^k_{\Omega}(M) \rightarrow H^k(M; \mathbb{R})$$ is an isomorphism.
\end{theorem}

\subsection{The Proof}

We begin with some definitions.
As usual, let $M$ denote a smooth manifold of dimension $n$. 
We call $M$ a de Rham manifold when, for all $k \geq 0$, the homomorphism $\mathscr{I}: H^k_{\Omega}(M) \rightarrow H^k(M; \mathbb{R})$ is an isomorphism.

For any smooth manifold $M$, we define an open cover $\{U_i\}$ of $M$ as a de Rham cover if each subset $U_i$ is itself a de Rham manifold, and if every finite intersection $U_{i_1} \, \cap U_{i_2} \, \cap \, \ldots \, \cap \, U_{i_k}$ is also de Rham. A de Rham cover that serves as a basis for the topology of $M$ is referred to as a de Rham basis for $M$.

\begin{lemma}{Lemma \ref*{section:derhamtheorem}.5: Disjoint Union is de Rham}
    If $\{M_j\}$ is any countable collection of de Rham manifolds, then their disjoint union is de Rham.
\end{lemma}
\begin{proof}
    Let $M = \coprod_j M_j$, where $\{M_j\}$ be a countable collection of smooth $n$-manifolds with or without boundary.
    
    The inclusion maps $\iota_j : M_j \hookrightarrow M$ induce an isomorphism from $H^k(M; \mathbb{R})$ to $\prod_j H^k(M_j; \mathbb{R})$ (\cite{lee2012smooth}, Proposition 18.5).
    Similarly, the inclusion maps $\iota_j$ induce an isomorphism from $H^k_{\Omega}(M)$ to $\prod_j H^k_{\Omega}(M_j)$, for each $k \geq 0$.
    To see why, note the pullback maps $\iota_j^* : \Omega^k(M) \rightarrow \Omega^k(M_j)$ induce an isomorphism from $\Omega^k(M)$ to $\prod_j \Omega^k(M_j)$ defined by
    \[
    \omega : (\iota_1^* \omega, \iota_2^* \omega, \dots) = (\omega|_{M_1}, \omega|_{M_2}, \dots)
    \]
     The given map is injective because any smooth $k$-form that is zero when restricted to each $M_j$ must be zero itself. Additionally, it is surjective because specifying an arbitrary smooth $k$-form on each $M_j$ uniquely defines one on the entire manifold $M$.
    
    As a result, for both de Rham and singular cohomology, these inclusions induce isomorphisms between the cohomology groups of the disjoint union and the direct product of the cohomology groups of the manifolds $M_j$. Furthermore, by the naturality of the de Rham homomorphism (Lemma \ref*{section:derhamtheorem}.3), $\mathscr{I}$ commutes with these isomorphisms.
\end{proof}

\begin{lemma}{Lemma \ref*{section:derhamtheorem}.6}
    Every convex open subset of $\mathbb{R}^n$ is de Rham.
\end{lemma}
\begin{proof}
    Let $U$ be a convex open subset of $\mathbb{R}^n$. 

    \vskip 0.1in

    \noindent\textbf{Case 1: $k > 0$}

    \noindent Homotopy equivalent spaces have isomorphic singular cohomology groups. For any one-point space $\{q\}$, $H^k(\{q\}; \mathbb{R})$ is trivial for $k > 0$ (\cite{lee2012smooth}, Proposition 18.5).
    Since $U$ is homotopy equivalent to a one-point space, the singular cohomology groups of $U$ are also trivial for $k > 0$. 
    Similarly, by the Poincaré lemma, $H^k_{\Omega}(U)$ is trivial for $k > 0$.
    
    \vskip 0.1in

    \noindent\textbf{Case 2: $k = 0$}

    \noindent By the Poincaré lemma, $H^0_{\Omega}(U)$ is the 1-dimensional space consisting of the constant functions (which is generated by the constant function $c_1: M \rightarrow \mathbb{R}$ defined by $c_1(p) = 1$).
    Likewise, $H^0(U; \mathbb{R}) \cong \text{Hom}(H^0(U), \mathbb{R})$ is also 1-dimensional because $H^0(U)$ is generated by the dual of any singular 0-simplex $\sigma: \Delta_0 = \{0\} \rightarrow M$ (which is smooth because any map from a 0-manifold is smooth). Then,
    \begin{align*}
        \mathscr{I}[c_1][\sigma] &= \int_{\Delta_0} \sigma^* c_1 \\
        &= (c_1 \circ \sigma)(0) \\
        &= 1
    \end{align*}
    Thus, $\mathscr{I}: H^0_{\Omega}(U) \to H^0(U; \mathbb{R})$ is nontrivial, so it is an isomorphism.
\end{proof}

\noindent We are finally ready to prove the main theorem.

\begin{proof} \textit{(de Rham's Theorem)}

        \vskip 0.1in

        \noindent It suffices to show that every smooth manifold is de Rham.

        \subsubsection{Part 1}
        \textbf{If $M$ has a ﬁnite de Rham cover, then $M$ is de Rham.}

        \noindent Let $M = U_1 \cup \ldots \cup U_k$, where the open subsets $U_i$ and their finite intersections are de Rham. We proceed by induction on $k$. 

        \vskip 0.1in

        \noindent \textit{Base Case}:
        
        \noindent For $k = 1$, clearly, $M = U_1$ is de Rham.

        \vskip 0.1in

        \noindent\textit{Inductive Step}:
        
        \noindent Suppose that $M$ has a de Rham cover consisting of two sets $\{U, V\}$. By combining the Mayer–Vietoris sequences for de Rham and singular cohomology, we construct the following commutative diagram, in which the horizontal rows are exact sequences and the vertical maps correspond to de Rham homomorphisms:

        \vskip 0.2in

        \noindent\adjustbox{scale=0.8,center}{
        \begin{tikzcd}
        {H_{\Omega}^{k-1}(U) \oplus H_{\Omega}^{k-1}(V) } & {H_{\Omega}^{k-1}(U \cap V)} & {H_{\Omega}^{k}(M)} & {H_{\Omega}^{k}(U) \oplus H_{\Omega}^{k}(V)} & {H_{\Omega}^{k}(U \cap V)} \\
        	\\
        	{H^{k-1}(U; \mathbb{R}) \oplus H^{k-1}(V; \mathbb{R}) } & {H^{k-1}(U \cap V; \mathbb{R})} & {H^k(M; \mathbb{R})} & {H^{k}(U; \mathbb{R}) \oplus H^{k}(V; \mathbb{R})} & {H^{k}(U \cap V; \mathbb{R})}
        	\arrow["{{\mathscr{I} \oplus \mathscr{I}}}"', from=1-1, to=3-1]
        	\arrow["{{j^*}}", from=1-1, to=1-2]
        	\arrow["{{j^*}}", from=3-1, to=3-2]
        	\arrow["{{\mathscr{I}}}"', from=1-2, to=3-2]
        	\arrow["{{\partial^*}}", from=3-2, to=3-3]
        	\arrow["\delta", from=1-2, to=1-3]
        	\arrow["{{\mathscr{I}}}"', from=1-3, to=3-3]
        	\arrow["{{i^*}}", from=3-3, to=3-4]
        	\arrow["{{i^*}}", from=1-3, to=1-4]
        	\arrow["{{\mathscr{I} \oplus \mathscr{I}}}"', from=1-4, to=3-4]
        	\arrow["{{j^*}}", from=1-4, to=1-5]
        	\arrow["{{j^*}}", from=3-4, to=3-5]
        	\arrow["{{\mathscr{I}}}"', from=1-5, to=3-5]
        \end{tikzcd}          
        }

        \vskip 0.2in

        Due to the naturality of the de Rham Homomorphism, the diagram commutes.
        Applying the inductive hypothesis, we know the first, second, fourth, and fifth vertical maps are all isomorphisms. Consequently, by the Five Lemma, we can conclude that the middle map is also an isomorphism. Thus, $M$ is a de Rham manifold.

        \vskip 0.1in

        Now, assume the claim holds for smooth manifolds with a de Rham cover of $k \geq 2$ open sets. Consider a de Rham cover ${U_1, \ldots, U_{k+1}}$ of $M$. Define $U$ as the union of the first $k$ sets, i.e., $U = U_1 \cup \ldots \cup U_k$, and set $V = U_{k+1}$.

        \vskip 0.1in

        By the inductive hypothesis, $U$ and $V$ are de Rham.
        Since $U \cap V$ admits a $k$-fold de Rham cover: ${U_1 \cap U_{k+1}, \ldots, U_k \cap U_{k+1}}$, it is also de Rham. Therefore, $M = U \cup V$ is a de Rham manifold.

        \subsubsection{Part 2}
        \textbf{If $M$ has a de Rham basis, then $M$ is de Rham.}

        \noindent Suppose $\{U_\alpha\}$ is a de Rham basis for $M$. 
        Recall that for any topological space $M$, an exhaustion function is a continuous function $f: M \rightarrow \mathbb{R}$ such that for each $c \in \mathbb{R}$, the set $f^{-1}((-\infty, c])$ is compact. Every smooth manifold, with or without boundary, admits a smooth, positive exhaustion function (\cite{lee2012smooth}, Proposition 2.28).

        \vskip 0.1in
        
        Now, let $f: M \to \mathbb{R}$ be such an exhaustion function.
        For each $m \in \mathbb{Z}_{\geq 0}$, define subsets $A_m$ and $A'_{m}$ of $M$ by
        \[
        A_m = f^{-1}\left( \left [m, m + 1 \right] \right) = \left\{ q \in M \mid m \leq f(q) \leq m + 1 \right\}
        \]
        and
        \[
        A'_{m} = f^{-1}\left( \left[ m - \frac{1}{2}, m + \frac{3}{2} \right] \right) = \left\{ q \in M \mid m - \frac{1}{2} < f(q) < m + \frac{3}{2} \right\}
        \]
        \begin{center}
            \includegraphics[width=0.5\textwidth]{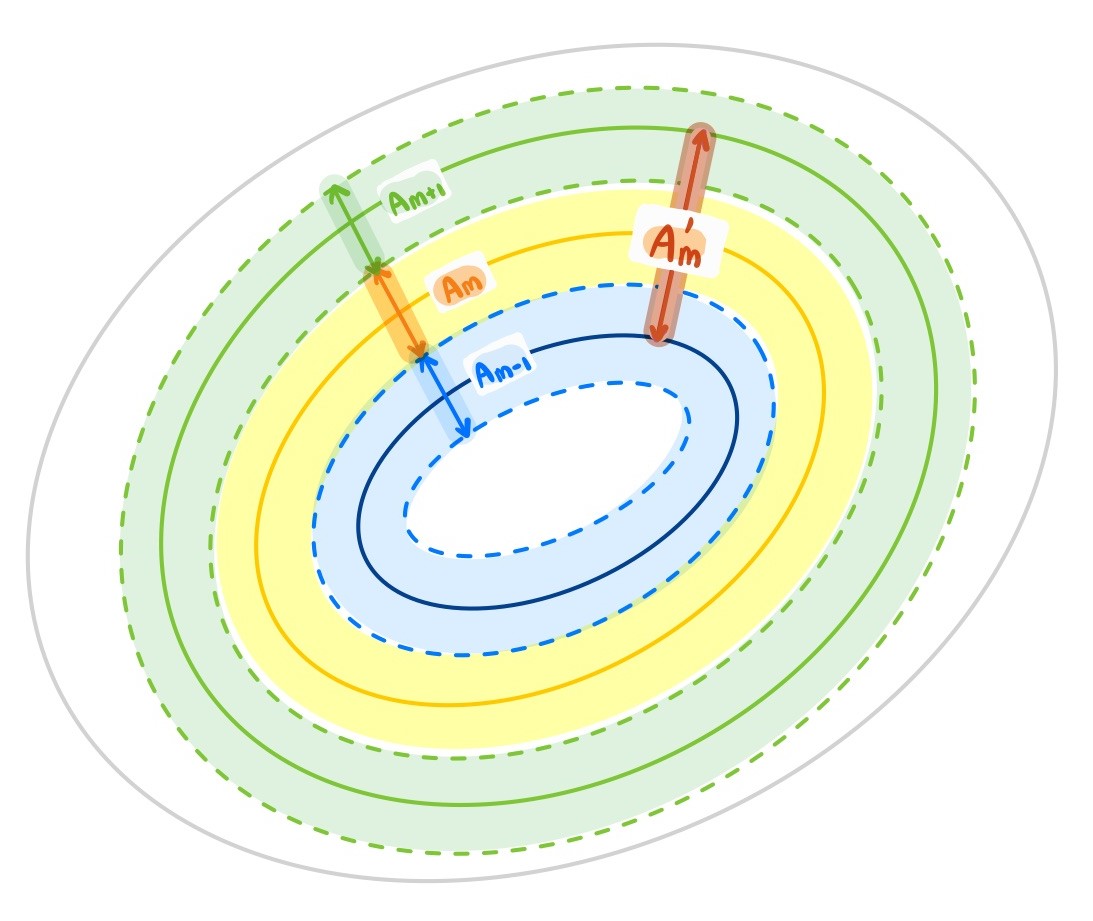}
        \end{center}
        We have that $$M = \bigcup_{m \geq 0} A_{m}$$
        Recall $\{U_\alpha\}$ is a de Rham basis, so $A'_{m}$ can be written as a union of basis elements.
        Since $A_{m} \subseteq A'_{m}$, for each point $p$ in $A_{m}$ there is an open set $U_i \in \{U_\alpha\}$ with $p \in U_i$ and $U_i \subseteq A'_{m}$.
        Thus, $A_{m}$ can be written as a union of basis elements.

        \vskip 0.1in
        
        The set of all such basis elements is an open cover of $A_m$.
        Furthermore, since $f$ is an exhaustion function, $A_m$ is compact.
        Therefore, this open cover is finite.
        Let $\beta_m$ be the union of this finite collection of sets. This is itself a finite de Rham cover of $\beta_m$, so by Part 1, $\beta_m$ is de Rham.  

        \vskip 0.1in

        Since $\beta_m \subseteq A'_{m}$, the nonempty intersection $\beta_m \cap \beta_{m'} \neq \emptyset$ implies $m' \in \{ m - 1, m, m+1 \}$.
        Therefore, 
        $$U = \bigcup_{m \text{ odd}} B_m$$
        and 
        $$V = \bigcup_{m \text{ even}} B_m$$ 
        are disjoint unions of de Rham manifolds.
        Recall if $\{Mj\}$ is any countable collection of de Rham manifolds, then their disjoint union is de Rham (Lemma \ref*{section:derhamtheorem}.5).
        Therefore, $U$ and $V$ are both de Rham.

        \vskip 0.1in

        Lastly, consider $U \cap V = \bigcup_{m \in \mathbb{Z}} B_m \cap B_{m+1}$, each of which has a finite de Rham cover consisting of sets of the form $U_\alpha \cap U_\beta$.
        Since it is a disjoint union of such sets, $U \cap V$ is de Rham.
        Therefore, $M = U \cup V$ is de Rham by Part 1.

        \subsubsection{Part 3}
        \textbf{Every open subset of $\mathbb{R}^n$ is de Rham.}

        \noindent Suppose $U \subseteq \mathbb{R}^n$ is an open subset.
        By definition, $U$ has a basis comprised of open Euclidean balls $B_r(p) = \{x \in \mathbb{R}^n \mid d(x, p) < r\}$ 
        Since each ball is convex, it is a de Rham manifold, and as any finite intersection of balls remains convex, the finite intersections are de Rham. Consequently, $U$ has a de Rham basis and is therefore de Rham by Part 2.

        \subsubsection{Part 4}
        \textbf{Every smooth manifold is de Rham. }

        \noindent Every smooth manifold can be equipped with a basis of atlas charts, each of which is diffeomorphic to an open subset of $\mathbb{R}^n$, as are their finite intersections.
        Given the naturality of the de Rham Homomorphism, it follows that $\mathscr{I}$ commutes with the cohomology maps induced by smooth maps. Consequently, any manifold that is diffeomorphic to a de Rham manifold is also de Rham.
        Hence, the basis of atlas charts forms a de Rham basis. 

        \vskip 0.1in

        Consequently, every smooth manifold has a de Rham basis and, by Part 2, is therefore de Rham.
\end{proof}




\newpage
\bibliography{refs}

\vskip 0.2in

\noindent Mathematical Institute, University of Oxford, OX2 6GG, UK \\
\noindent \textit{Email Address}: \url{alice.petrov@maths.ox.ac.uk}


\end{document}